\documentclass[11pt]{article}
\usepackage[margin=1in]{geometry}
\usepackage{graphicx,xcolor,cite,mathtools}
\usepackage{amssymb,amsmath,amsthm,mathrsfs,paralist,bm,esint,setspace}
\RequirePackage[colorlinks,citecolor=blue,urlcolor=blue]{hyperref}

\newcommand{\R}{{\mathbb{R}}}
\newcommand{\E}{\mathrm{E}}

\renewcommand{\P}{\mathrm{P}}
\renewcommand{\d}{\mathrm{d}}

\newcommand{\e}{\mathrm{e}}

\DeclareMathOperator{\Cov}{\text{\rm Cov}}

\DeclarePairedDelimiter\floor{\lfloor}{\rfloor}
\newcommand{\ca}{\mathcal{A}}
\newcommand{\vep}{\varepsilon}
\newcommand{\cl}{\mathcal{L}}
\renewcommand{\P}{\mathbb{P}}

\newcommand{\cd}{\mathcal{D}}

\newcommand{\ce}{\mathcal{E}}

\newcommand{\Z}{\mathbb{Z}}
\newcommand{\N}{\mathbb{N}}

\newcommand*{\bo}{\boldsymbol}

\newcommand{\cb}{\mathcal{B}}
\newcommand{\cc}{\mathcal{C}}

\title{Macroscopic Hausdorff dimension of the level sets of the Airy processes}
\author{ Sudeshna Bhattacharjee and Fei  Pu
	}
\date{}                                           

\begin{document}
\newtheorem{stat}{Statement}[section]
\newtheorem{proposition}[stat]{Proposition}
\newtheorem*{prop}{Proposition}
\newtheorem{corollary}[stat]{Corollary}
\newtheorem{theorem}[stat]{Theorem}
\newtheorem{lemma}[stat]{Lemma}
\theoremstyle{definition}
\newtheorem{definition}[stat]{Definition}
\newtheorem*{cremark}{Remark}
\newtheorem{remark}[stat]{Remark}
\newtheorem*{OP}{Open Problem}
\newtheorem{example}[stat]{Example}
\newtheorem{nota}[stat]{Notation}
\numberwithin{equation}{section}
\maketitle

\begin{abstract}
          We study the Macroscopic Hausdorff dimension of the upper and lower level sets of the Airy processes, 
          following the general method developed in Khoshnevisan et al. \cite{KKX17}.  For the Airy$_1$ process, the approach to  macroscopic Hausdorff dimension of level sets hinges on 
          some inequalities for its joint probabilities, while for the Airy$_2$ process, we make use of some quantitative estimates on the tail probabilities of its maximum and minimum over an interval. 
\end{abstract}

\bigskip\bigskip

\noindent{\it \noindent MSC 2010 subject classification: 60A10, 60F05, 60H15.}
 \\
\noindent{\it Keywords: Airy processes, level sets, macroscopic Hausdorff dimension.}


{
}

\section{Introduction}

The Airy$_1$ and Airy$_2$ processes are introduced by Sasamoto \cite{Sas05} and Pr\"ahofer and Spohn \cite{PSpng02} respectively in the context of random growth models lying in the KPZ universality class.  They are stationary stochastic processes whose finite-dimensional distributions are given in terms of the Fredholm determinant (see \cite[Section 4.1]{WFS17}).  Recently, the limit theorems for the Airy processes have been studied in \cite{Pu23} and \cite{BaB24}. 
According to \cite[Theorem 1.4]{Pu23} and  \cite[Theorem 1.1(i)]{BaB24}, we see from stationarity of the Airy$_1$ and Airy$_2$ processes that
          \begin{align}
          \limsup_{t\to\infty} \frac{\mathcal{A}_1(t)}{\frac12((3\log t)/2)^{2/3}} &= 1, \quad \text{a.s.}\label{eq:sup}
          \\
                    \limsup_{t\to\infty} \frac{\mathcal{A}_2(t)}{((3\log t)/4)^{2/3}} &= 1, \quad \text{a.s.}\label{eq:sup2}
          \end{align}

The limits in \eqref{eq:sup} and \eqref{eq:sup2} indicate that $t\mapsto\frac12((3\log t)/2)^{2/3}$ and $t\mapsto ((3\log t)/4)^{2/3}$ are gauge functions for measuring the tall peaks of the random height functions $t\mapsto \mathcal{A}_1(t)$ and $t\mapsto \mathcal{A}_2(t)$, respectively.  We are interested in the upper level sets of the Airy processes and consider the random sets
\begin{align*}
\mathcal{U}_1(\gamma)&:=\left\{t>\e: \mathcal{A}_1(t)> \frac{\gamma}{2}((3\log t)/2)^{2/3} \right\}, \\
\mathcal{U}_2(\gamma)&:=\left\{t>\e: \mathcal{A}_2(t)> \gamma((3\log t)/4)^{2/3} \right\},
\end{align*}
for $\gamma>0$.

Denote by ${\rm Dim_H}(E)$ the (Barlow-Taylor)  macroscopic Hausdorff dimension of $E\subset \R$ (see Section \ref{sec:pre} for the definition and related information). Our first goal is to determine the macroscopic Hausdorff dimension of the upper level sets of the Airy processes.

\begin{theorem}\label{th:Hausdorff}
           For $\gamma\in(0, 1)$,
           \begin{align}
           {\rm Dim_{H}}(\mathcal{U}_1(\gamma))&=1-\gamma^{3/2}, \quad \text{a.s.} \label{eq:dim}\\
           {\rm Dim_{H}}(\mathcal{U}_2(\gamma))&=1-\gamma^{3/2}, \quad \text{a.s.}\label{eq:dimU2}
           \end{align}           
\end{theorem}
The notion of macroscopic Hausdorff dimension is due to Barlow and Taylor \cite{BaT89, BaT92}, which describes the large-scale geometry of a set.  Khoshnevisan et al. \cite{KKX17} make use of the macroscopic Hausdorff dimension to study the 
multifractal behavior of the peaks of the solution to stochastic PDEs. In light of the definition of multifractality in \cite[Definition 1.1]{KKX17}, Theorem \ref{th:Hausdorff} implies that the tall peaks of the Airy processes are multifractal 
with respect to their corresponding gauge functions.
We refer to \cite{DaP23,GhY23, KKX18, Yi23} for the multifractal properties of the peaks for some random models related to stochastic PDEs.

We are also interested in the fractal properties of the valleys of the Airy processes.  Basu and Bhattacharjee \cite{BaB24} 
have recently obtained precise limits for the asymptotic behavior of the minimum of the Airy processes. According to \cite[Theorem 1.1(ii), Theorem 1.2(ii)]{BaB24} and by the stationarity of  the Airy processes, we have
\begin{align*}
          &\liminf_{t\to\infty} \frac{\mathcal{A}_1(t)}{(3\log x)^{1/3}} =-1, \quad \text{a.s.}\\
          &\liminf_{t \rightarrow \infty} \frac{\ca_2(t)}{(12 \log t)^{1/3}}=-1, \quad \text{a.s.}
          \end{align*}
This motivates us to consider the lower level sets of the Airy processes         
\begin{align*}
&\mathcal{L}_1(\gamma):=\left\{t>\e: \mathcal{A}_1(t)< -\gamma(3\log t)^{1/3} \right\}, \\
&\mathcal{L}_2(\gamma):=\left \{t>\e:\ca_2(t)<-\gamma(12 \log t)^{1/3} \right \},
\end{align*}
for $\gamma>0$.

Analogous to Theorem \ref{th:Hausdorff}, we have the following. 

\begin{theorem}\label{th:Hausdorff2}
           For $\gamma\in(0, 1)$,
           \begin{align}
           {\rm Dim_{H}}(\mathcal{L}_1(\gamma))&=1-\gamma^{3}, \quad \text{a.s.} \label{eq:dimL1}\\
           {\rm Dim_{H}}(\mathcal{L}_2(\gamma))&=1-\gamma^{3}, \quad \text{a.s.} \label{eq:dimL2}
           \end{align}           
\end{theorem}
Theorem \ref{th:Hausdorff2} suggests that the valleys of the Airy processes are multifractal 
with respect to their corresponding gauge functions. 
Note that the macroscopic Hausdorff dimensions of the upper and lower level sets are exactly the same for the Airy$_1$ and Airy$_2$ processes. This is because that the tail exponents are the same for the one-point distribution of both processes. We will make this point clear in the proofs of Theorems \ref{th:Hausdorff} and \ref{th:Hausdorff2}.

\section{Preliminary}\label{sec:pre}
             In this section, we first follow Barlow and Taylor \cite{BaT89, BaT92} and  Khoshnevisan et al. \cite{KKX17} to introduce
macroscopic Hausdorff dimension.  We will restrict to the one-dimensional case as the Airy processes have only one parameter.  For a Borel set $E\subset \R$, the $n$th shell of $E$ is defined as $E\cap \{(-\e^{n+1}, -\e^n]\cup [\e^n, \e^{n+1})\}$.
Fix a number $c_0>0$. For $\rho>0$, we define $\rho$-dimensional Hausdorff content of the $n$th shell of $E$ as
\begin{align*}
          \nu_{\rho}^n(E):= \inf \sum_{i=1}^m\left(\frac{{\rm length}(Q_i)}{\e^n}\right)^\rho,
\end{align*}
where the infimum is taken over all intervals $Q_1, \ldots, Q_m$ of length $\geq c_0$ that cover the $n$th shell of $E$. The Barlow–Taylor macroscopic Hausdorff dimension of $E\subset \R$ is defined as 
\begin{align*}
           {\rm Dim_{H}}(E):=\inf \left\{\rho>0: \sum_{n=1}^{\infty}\nu_\rho^n(E)<\infty\right\}.
\end{align*}
The macroscopic Hausdorff dimension does not depend on $c_0$ and hence we can choose $c_0=1$; see \cite[Lemma 2.3]{KKX17}.

Khoshnevisan et al. \cite{KKX17} introduced the notion of thickness of a set, an import tool to give a lower bound on the macroscopic Hausdorff dimension which we now recall.  Fix $\theta\in (0, 1)$ and define 
\begin{align} \label{eq:Pi}
\Pi_n(\theta):= \bigcup_{\substack{0\leq i \leq \e^{n(1-\theta)+1}-\e^{n(1-\theta)}\\  i\in \mathbb{Z}}}\left\{\e^n+i\e^{\theta n}\right\}.
\end{align}

\begin{definition}[{{\cite[Definition 4.3]{KKX17}}}] \label{def:thick}
We say that $E\subset \R$ is $\theta$-thick if there exists an integer $M=M(\theta)$ such that 
$          E \cap [x, x+\e^{\theta n}) \neq \emptyset
$
for all $x\in \Pi_n(\theta)$ and $n\geq M$. 
\end{definition}

The following criterion on the lower bound of macroscopic Hausdorff dimension is due to Khoshnevisan et al. \cite{KKX17}. 

\begin{proposition}[{{\cite[Proposition 4.4]{KKX17}}}] \label{prop:thick}
          If $E\subset \R $ is $\theta$-thick for some $\theta\in (0, 1)$, then ${\rm Dim_{H}}(E)\geq 1-\theta$.
\end{proposition}

We will study the macroscopic Hausdorff dimension of the Airy$_1$ process using the association property.  Recall from Esary et al. \cite{EPW67} that
a random vector $X:=(X_1\,,\ldots,X_m)$ is said to be \emph{associated} if
	\begin{equation}\label{E:assoc}
		\Cov[h_1(X)\,,h_2(X)]\ge0,
	\end{equation}
for every pair of functions $h_1,h_2:\R^m\to\R$
that are nondecreasing in every coordinate and satisfy $h_1(X),h_2(X)\in L^2(\Omega)$. A random field $\Phi=\{\Phi(x)\}_{x\in\R^d}$ is
	\emph{associated} if $(\Phi(x_1)\,,\ldots,\Phi(x_m))$ is associated
	for every $x_1,\ldots,x_m\in\R^d$.
We remark that an associated random vector is also called to satisfy the FKG inequalities; see Newman \cite{New80}.

We recall some useful probability inequalities for an associated random vector $(X_1\,,\ldots,X_m)$. Let $x_1, \ldots, x_m$ be real numbers. Then 
          \begin{align}\label{FKG1}
          &\mathrm{P}\{X_j\leq x_j, 1\leq j\leq m\} - \prod_{j=1}^m\mathrm{P}\{X_j\leq x_j\}  \nonumber\\
            &\qquad \qquad \qquad \qquad\qquad \leq\sum_{1\leq j<k\leq m} 
          \left(\mathrm{P}\{X_j\leq x_j, X_k\leq x_k\} - \mathrm{P}\{X_j\leq x_j\}\mathrm{P}\{X_k\leq  x_k\}\right).
                     \end{align}
The above inequality is proved in \cite[Lemma 2.1]{Pu23},  based on the Lebowitz's inequality (see \cite[Theorem 1.2.2]{PR12} and \cite{Leb72}).
Similarly, using the Lebowitz's inequality, we can derive that 
          \begin{align}\label{FKG2}
          &\mathrm{P}\{X_j> x_j, 1\leq j\leq m\} - \prod_{j=1}^m\mathrm{P}\{X_j> x_j\}  \nonumber\\
            &\qquad \qquad \qquad \qquad\qquad \leq\sum_{1\leq j<k\leq m} 
          \left(\mathrm{P}\{X_j> x_j, X_k> x_k\} - \mathrm{P}\{X_j> x_j\}\mathrm{P}\{X_k> x_k\}\right) \nonumber\\
           &\qquad \qquad \qquad \qquad\qquad =\sum_{1\leq j<k\leq m} 
          \left(\mathrm{P}\{X_j\leq x_j, X_k\leq x_k\} - \mathrm{P}\{X_j\leq x_j\}\mathrm{P}\{X_k\leq x_k\}\right),
          \end{align}
          where the equality holds obviously.  Because the Airy$_1$ process is associated (see \cite[Theorem 1.2]{Pu23}), its marginal satisfies the above two inequalities. Moreover, since the one-point
          distribution of the Airy$_1$ process, i.e., the GOE Tracy-Widom distribution has bounded and continuous density and finite second moment, we deduce from 
           \cite[(6.2.20)]{PR12}  that 
           there exists a constant $K>0$ such that for all $x, y\in\R$,
          \begin{align}\label{eq:prob}
          &\sup_{s, t\in \R}\left(\mathrm{P}\left\{\mathcal{A}_1(x) \leq s, \mathcal{A}_1(y)\leq t\right\} - 
          \mathrm{P}\left\{\mathcal{A}_1(x) \leq s\right\}\mathrm{P}\left\{\mathcal{A}_1(y)\leq t\right\}  \right)           \leq K\left[\Cov(\mathcal{A}_1(x)\,,
          \mathcal{A}_1(y))\right]^{1/3}
          \end{align}
          (see also \cite[(6.11)]{Pu23}).
          
          Using the fact that certain centered and scaled passage times in the exponential last passage percolation (LPP) converge to the Airy processes, we will appeal to the exponential LPP to study the tail probability of maximum and minimum of the Airy processes. We state here the weak convergence results.
           Let us introduce some notations of the relevant exponential LPP model. 
Consider the last passage percolation on $\mathbb{Z}^2$ with i.i.d. $\rm Exp(1)$ passage times on the vertices.  For $x\in \R$, let $u=u_N(s)= (N-\floor{s(2N)^{2/3}}, N+\floor{s(2N)^{2/3}})\in \mathbb{Z}^2$. Denote by $\cl_r$  the line $\{(x,y) \in \Z^2:x+y=r\}$.
Let $T_{N}(s):=T_{\mathbf{0},u}$ denote the last passage time from $\bo{0}:=(0,0)$ to $u$ (i.e., the maximal total weight of an up/right path connecting $\bo{0}$ and $u$ excluding the last vertex).
Let $T_N^*(s)$ denote the last passage time from $\mathcal{L}_0$ to $u$ (i.e., the maximal weight among all paths that start at some point in $\mathcal{L}_0$ and end at $u$ excluding the last vertex).  We denote by $\Gamma_{u,v}$ the almost surely unique geodesic (i.e, the up/right path with maximal passage time) between two points $u,v \in \Z^2$. $\Gamma_{v}$ will be used to denote the geodesic between $\bo{0}$ and $v$.


The following results on weak convergence of exponential LPP are taken from \cite{BaB24} and \cite{BGZ21}.

\begin{theorem}[{\cite[Theorem 3.8]{BGZ21}}]
    \label{t:lppairy2limit}
    As $N\to \infty$, 
    \begin{equation}
    \label{e:airy2limit}
    \frac{T_{N}(s)-4N}{2^{4/3}N^{1/3}}\Rightarrow \ca_2(s)-s^2,
\end{equation}
where $\Rightarrow$ denotes weak convergence in the topology of uniform convergence on compact sets.
\end{theorem}

\begin{theorem}[{{\cite[Theorem 1.4]{BaB24}}}]\label{th:weak}
          As $N\to\infty$,
          \begin{align*}
          \frac{T_N^*(s)-4N}{2^{4/3}N^{1/3}} \Rightarrow  2^{1/3}\mathcal{A}_1(2^{-2/3}s).
          \end{align*}
          \end{theorem}
          


\section{Macroscopic Hausdorff dimension of the Airy$_1$ process} \label{sec:th1}

We first recall that the one-point distribution of the Airy$_1$ process is given by
\begin{align*}
\mathrm{P}\{\mathcal{A}_1(0)\leq x\} =F_1(2x), \quad x\in\R,
\end{align*}
where $F_1$ denotes the GOE Tracy-Widom distribution. By the asymptotic behavior of the GOE Tracy-Widom distribution (see \cite{BBD08, DuV13}), we see that as $x\to+\infty$
\begin{align}
          \mathrm{P}\{\mathcal{A}_1(0)> x\}&= \e^{-(\frac{4\sqrt{2}}{3}+o(1))x^{3/2}}, \label{eq:uppertail}\\
          \mathrm{P}\{\mathcal{A}_1(0)< -x\}&= \e^{-(\frac13+o(1))x^{3}}.\label{eq:lowertail2}
\end{align}
Moreover, according to  \cite[Proposition 6.1]{Pu23}, we have as $x\to+\infty$,
\begin{align}\label{eq:max-tail}
            \mathrm{P}\left\{\max_{s\in [0,1]}\mathcal{A}_1(s)> x\right\}&= \e^{-(\frac{4\sqrt{2}}{3}+o(1))x^{3/2}}.
\end{align}
Furthermore, Basu et al. \cite{BBF23} have shown that the covariance of the Airy$_1$ process 
 decays super-exponentially by showing that there exists $c'>0$ such that for all $t>1$,
 \begin{align}\label{eq:cov}
 \Cov(\mathcal{A}_1(t)\,, \mathcal{A}_1(0))\leq \e^{c't^2}\e^{-\frac43t^3},
 \end{align}
(see \cite[Theorem 1.1]{BBF23}).

We first establish the macroscopic Hausdorff dimension of the upper level sets of the Airy$_1$ process in Theorem \ref{th:Hausdorff}. 

\begin{proof}[{Proof of \eqref{eq:dim}}]
          Theorems 4.1 and 4.7 of \cite{KKX17} provide general macroscopic Hausdorff dimension estimates.
           First, since the Airy$_1$ process is stationary, 
          we can combine \cite[Theorem 4.1]{KKX17} with \eqref{eq:max-tail} to deduce that for $\gamma\in (0,1)$
           \begin{align}\label{eq:dim-upper}
           {\rm Dim_{H}}\left\{t>\e: \mathcal{A}_1(t) >\frac\gamma2\left((3\log t)/2\right)^{2/3}\right\}\leq 1-\gamma^{3/2}, \quad \text{a.s.}
           \end{align}

           Theorem 4.7 of \cite{KKX17} cannot be applied directly to give a lower bound on the macroscopic Hausdorff dimension since we are not 
           able to construct a coupling process for the Airy$_1$ process which satisfies the conditions of Theorem 4.7 of \cite{KKX17}. 
           Instead, we will use the probability inequalities in Section \ref{sec:pre} for the Airy$_1$ process to show that the upper level set is $\theta$-thick and 
           then derive a lower bound for the Hausdorff dimension by Proposition \ref{prop:thick}.
            
            Fix  $\gamma\in (0,1)$. We are aiming  to show that 
           \begin{align} \label{eq:dim-lower}
           {\rm Dim_{H}}\left\{t>\e: \mathcal{A}_1(t) >\frac\gamma2\left((3\log t)/2\right)^{2/3}\right\}\geq 1-\gamma^{3/2}, \quad \text{a.s.}
           \end{align} 
         We will prove that 
                              \begin{align*}
                      {\rm Dim_{H}}\left\{t>\e: \mathcal{A}_1(t) >\frac\gamma2\left((3\log t)/2\right)^{2/3}\right\} \geq 1-\theta, \quad
           \text{a.s. for all $\theta\in(\gamma^{3/2}, 1)$}.
           \end{align*} 
           Choose and fix $\theta\in (\gamma^{3/2}, 1)$. We also choose and fix sufficiently small positive constants $\delta$ and $\eta$ such that
           \begin{align}\label{exponent2}
           \theta-\delta-(1+\frac{3}{4\sqrt{2}}\eta)\gamma^{3/2}>0. 
           \end{align}

           Denote 
           \begin{align}\label{x_in}
           x_{i, n}= \e^{n}+(i-1)\e^{n\theta}
           \end{align}
            for $i=1, \ldots, L_n:=\floor{(\e-1)\e^{n(1-\theta)}+1}$.  Recall the set $\Pi_n(\theta)$ defined in \eqref{eq:Pi} and we write
           \begin{align*}
           \Pi_n(\theta)= \{x_{i, n}: i=1, \ldots, L_n\}.
           \end{align*}
           For $i=1, \ldots, L_n$, we define 
           \begin{align}\label{z_in}
           z_{j, n}(i)= x_{i, n}+(j-1)\e^{n\delta}
           \end{align}
            where $j=1, \ldots, \ell_n(i)$ 
           with 
           \begin{align}
           \ell_n(i):=\floor{\e^{n(\theta-\delta)}}+1\in (\e^{n(\theta-\delta)}, 2\e^{n(\theta-\delta)}). \label{eq:ln}
           \end{align}  
           Since $\log s \leq n+1$ for all $s\in [x_{i,n}, x_{i,n}+\e^{n\theta})$ with $i=1, \ldots, L_n$, we write
           \begin{align}\label{eq:sum}
           \P\left\{\sup_{s\in [x_{i,n}, x_{i,n}+\e^{n\theta})}\frac{\mathcal{A}_1(s)}{(\log s)^{2/3}} \leq 
           \frac\gamma2\left(3/2\right)^{2/3}\right\} 
           &\leq            \P\left\{\sup_{s\in [x_{i,n}, x_{i,n}+\e^{n\theta})}\mathcal{A}_1(s)\leq 
           \frac\gamma2\left(3(n+1)/2\right)^{2/3}\right\} \nonumber\\
           & \leq             \P\left\{\max_{1\leq j\leq \ell_n(i)}\mathcal{A}_1(z_{j,n}(i))\leq 
           \frac\gamma2\left(3(n+1)/2\right)^{2/3}\right\}\nonumber\\
           &=           \mathcal{P}_{i,n,1}+           \mathcal{P}_{i,n,2},
                      \end{align}
           where
           \begin{align*}
           \mathcal{P}_{i,n,1}&=             \P\left\{\max_{1\leq j\leq \ell_n(i)}\mathcal{A}_1(z_{j,n}(i))\leq 
           \frac\gamma2\left(3(n+1)/2\right)^{2/3}\right\} - \prod_{j=1}^{\ell_n(i)}\P\left\{\mathcal{A}_1(z_{j,n}(i))\leq 
           \frac\gamma2\left(3(n+1)/2\right)^{2/3}\right\},\\
           \mathcal{P}_{i,n,2}&=\prod_{j=1}^{\ell_n(i)}\P\left\{\mathcal{A}_1(z_{j,n}(i))\leq 
           \frac\gamma2\left(3(n+1)/2\right)^{2/3}\right\}.
           \end{align*}
           Since the Airy$_1$ process is associated (see \cite[Theorem 1.2]{Pu23}), it follows from \eqref{FKG1} and \eqref{eq:prob} that
           \begin{align}\label{eq:dif}
            \mathcal{P}_{i,n,1}           &\leq K \sum_{1\leq j<k\leq  \ell_n(i)}\left[\Cov(\mathcal{A}_1(z_{j,n}(i))\,, \mathcal{A}_1(z_{k,n}(i)))\right]^{1/3}
           \nonumber\\
           &\leq K_1 \sum_{1\leq j<k\leq  \ell_n(i)} \e^{-\frac13|z_{j,n}(i)-z_{k,n}(i))|^3} \leq K_1  \ell_n(i)^2 \e^{-\frac13\e^{3n\delta}}\nonumber\\
           &\leq 4K_1\e^{2n(\theta-\delta)-\frac13\e^{3n\delta}},
           \end{align}
           where the second inequality holds by \eqref{eq:cov} and the last inequality is due to \eqref{eq:ln}.

           In order to estimate $\mathcal{P}_{i,n,2}$, we first notice from \eqref{eq:uppertail}
           that for all sufficiently large $n$, 
           \begin{align*}
           \P\left\{\mathcal{A}_1(0)\leq 
           \frac\gamma2\left(3(n+1)/2\right)^{2/3}\right\} &= 1-          \P\left\{\mathcal{A}_1(0)>
           \frac\gamma2\left(3(n+1)/2\right)^{2/3}\right\}\\
           &\leq 1-\e^{-(\frac{4\sqrt{2}}{3}+\eta)(\frac{\gamma}{2})^{3/2}\frac{3(n+1)}{2}} = 1- \e^{-(1+\frac{3\eta}{4\sqrt{2}})\gamma^{3/2}(n+1)},
           \end{align*}
           where $\eta$ is the fixed positive number satisfying \eqref{exponent2}. By stationarity of the Airy$_1$ process, 
           \begin{align}\label{eq:tail3}
           \mathcal{P}_{i,n,2}&= \left(\P\left\{\mathcal{A}_1(0)\leq 
           \frac\gamma2\left(3(n+1)/2\right)^{2/3}\right\} \right)^{\ell_n(i)}\nonumber\\
           &\leq  \left(1- \e^{-(1+\frac{3\eta}{4\sqrt{2}})\gamma^{3/2}(n+1)} \right)^{\ell_n(i)} \leq 
           \e^{-\ell_n(i)\e^{-(1+\frac{3\eta}{4\sqrt{2}})\gamma^{3/2}(n+1)} } \nonumber\\
           &\leq            \e^{-\e^{n(\theta-\delta)-(1+\frac{3\eta}{4\sqrt{2}})\gamma^{3/2}(n+1)} },
                      \end{align}
          where the second inequality holds by the inequality $1-x\leq \e^{-x}$ for all $x\geq0$ 
          and the third inequality by \eqref{eq:ln}.

           Now, using the fact $L_n\leq \e^{1+n(1-\theta)}$, we see from \eqref{eq:sum}, \eqref{eq:dif} and \eqref{eq:tail3} that
           \begin{align*}
           &\sum_{i=1}^{L_n}\P\left\{\sup_{s\in [x_{i,n}, x_{i,n}+\e^{n\theta})}\frac{\mathcal{A}_1(s)}{(\log s)^{2/3}} \leq 
           \frac\gamma2\left(3/2\right)^{2/3}\right\} \\
           &\qquad\qquad \qquad \qquad \leq 4K_1 \e^{1+n(1-\theta)+2n(\theta-\delta)-\frac13\e^{3n\delta}}
           + \e^{1+n(1-\theta)-\e^{n(\theta-\delta)-(1+\frac{3\eta}{4\sqrt{2}})\gamma^{3/2}(n+1)} }.
\end{align*}
           In light of \eqref{exponent2},  it follows that 
           \begin{align*}
           \sum_{n=1}^{\infty}\sum_{i=1}^{L_n}\P\left\{\sup_{s\in [x_{i,n}, x_{i,n}+\e^{n\theta})}\frac{\mathcal{A}_1(s)}{(\log s)^{2/3}} \leq 
           \frac\gamma2\left(3/2\right)^{2/3}\right\} <\infty.
           \end{align*}
          Borel-Cantelli's lemma ensures that almost surely, for all but a finite number of integers $n\geq1$
          \begin{align*}
          \sup_{s\in [x_{i,n}, x_{i,n}+\e^{n\theta})}\frac{\mathcal{A}_1(s)}{(\log s)^{2/3}} >
           \frac\gamma2\left(\frac{3}{2}\right)^{2/3}, \quad \text{for all $1\leq i\leq L_n$}.
          \end{align*}
          This implies that almost surely for all sufficiently large $n$, 
          \begin{align*}
          \mathcal{U}(\gamma):=\left\{t>\e: \mathcal{A}_1(t)> \frac{\gamma}{2}((3\log t)/2)^{2/3} \right\} \cap 
          [x, x+\e^{n\theta}) \neq \emptyset, \quad \text{for all $x\in \Pi_n(\theta)$}.
          \end{align*}
          In other words, the upper level set $\mathcal{U}(\gamma)$
           is $\theta$-thick almost surely (see Definition \ref{def:thick}). 
           Therefore, we conclude from Proposition \ref{prop:thick} that for $\theta \in (\gamma^{3/2}, 1)$
          \begin{align*}
          {\rm Dim_{H}}\left\{t>\e: \mathcal{A}_1(t)> \frac{\gamma}{2}((3\log t)/2)^{2/3} \right\} \geq1-\theta, \quad \text{a.s.}
          \end{align*}
          We let $\theta\downarrow\gamma^{3/2}$ to complete the proof of \eqref{eq:dim-lower}. 
          
          Therefore, the equality in \eqref{eq:dim} follows from \eqref{eq:dim-upper} and \eqref{eq:dim-lower}.
\end{proof}

We proceed to establish the macroscopic Hausdorff dimension of the lower level sets of the Airy$_1$ process in Theorem \ref{th:Hausdorff2}.  The proof of \eqref{eq:dimL1} is similar to that of \eqref{eq:dim}. In order to give an upper bound on the Hausdorff dimension of the lower level sets, we need to estimate the lower tail probability of the minimum of the Airy$_1$ process. 

\begin{lemma}\label{lem:min}
          For $\varepsilon\in(0,1)$, there exists a constant $C>0$ (depending on $\varepsilon$) such that for all sufficiently
          large $x$,
          \begin{align}
          \P\left\{\min_{s\in[0, 1]}\mathcal{A}_1(s) \leq -x\right\}\leq C\e^{-\frac13(1-\varepsilon)x^3}.
          \end{align}
\end{lemma}
\begin{proof}
          Fix $\varepsilon\in (0, 1)$. 
          By Lemma \ref{lem:LPP} below, there exists $\delta>0$  such that for any $x$ sufficiently large (depending on $\varepsilon$), for $N$ sufficiently large (depending on $x, \varepsilon$)
          \begin{align*}
          \P\left\{\min_{u\in I^\delta} T_{\mathcal{L}_0, u}-4N \leq -x 2^{4/3}N^{1/3}\right\} \leq \e^{-\frac16(1-\varepsilon)x^3}.
          \end{align*}
                    Since $\min_{s\in [-\delta/8, \delta/8]}T_N^*(s) \geq \min_{u\in I^\delta} T_{\mathcal{L}_0, u}$, it follows that
          \begin{align*}
                \P\left\{\min_{s\in [-\delta/8, \delta/8]}T_N^*(s)-4N \leq -x 2^{4/3}N^{1/3}\right\} \leq 
                  \P\left\{\min_{u\in I^\delta} T_{\mathcal{L}_0, u}-4N \leq -x 2^{4/3}N^{1/3}\right\} 
                  \leq \e^{-\frac16(1-\varepsilon)x^3}
          \end{align*}
                    for any $x$ sufficiently large (depending on $\varepsilon$) and for $N$ sufficiently large (depending on $x, \varepsilon$).           By Theorem \ref{th:weak} and the continuity of the mapping $f\mapsto \min_{s\in[-\delta/8, \delta/8]}f(s)$ in the topology of uniform convergence, we obtain that
          \begin{align*}
          \P\left\{2^{1/3}\min_{s\in [-\delta/8, \delta/8]}\mathcal{A}_1(2^{-2/3}s) \leq -x\right\} \leq \e^{-\frac16(1-\varepsilon)x^3}
          \end{align*}
                    for any $x$ sufficiently large (depending on $\varepsilon$). Using stationarity of the Airy$_1$ process, we see
           that for any $x$ sufficiently large (depending on $\varepsilon$),
           \begin{align*}
          \P\left\{\min_{s\in [0, \delta/4]}\mathcal{A}_1(2^{-2/3}s) \leq -x\right\} \leq \e^{-\frac13(1-\varepsilon)x^3}.
          \end{align*}
          If $2^{-2/3}\delta/4\geq1$, then 
          \begin{align*}
          \P\left\{\min_{s\in [0, 1]}\mathcal{A}_1(s) \leq -x\right\} \leq
          \P\left\{\min_{s\in [0, \delta/4]}\mathcal{A}_1(2^{-2/3}s) \leq -x\right\} \leq \e^{-\frac13(1-\varepsilon)x^3}.
          \end{align*}
          If $2^{-2/3}\delta/4<1$, then using again stationarity of the Airy$_1$ process,
          \begin{align*}
          \P\left\{\min_{s\in [0, 1]}\mathcal{A}_1(s) \leq -x\right\}& \leq
          \P\left\{\min_{s\in [0, \frac{\delta}{2^{2/3}4}\cdot (\floor{2^{2/3}4/\delta}+1)]}\mathcal{A}_1(s) \leq -x\right\} \\
          &\leq 
          (\floor{2^{2/3}4/\delta}+1)\P\left\{\min_{s\in [0, \frac{\delta}{4} ]}\mathcal{A}_1(2^{-2/3}s) \leq -x\right\}\\
          &
          \leq (\floor{2^{2/3}4/\delta}+1)\e^{-\frac13(1-\varepsilon)x^3}.
          \end{align*}
          Therefore, we conclude that there exists $C>0$ (depending on $\varepsilon$) such that for 
          any $x$ sufficiently large (depending on $\varepsilon$)
          \begin{align*}
          \P\left\{\min_{s\in[0, 1]}\mathcal{A}_1(s) \leq -x\right\}\leq C\e^{-\frac13(1-\varepsilon)x^3},
          \end{align*}
          which proves Lemma \ref{lem:min}.
\end{proof}

We are now ready to prove \eqref{eq:dimL1}. We first give an upper bound on the Hausdorff dimension of the lower level set of the Airy$_1$ process.

\begin{proof}[Proof of \eqref{eq:dimL1}: upper bound]
          The proof follows from a modification of the proof of \cite[Theorem 4.1]{KKX17} by using Lemma \ref{lem:min}.   
          We include the details for the sake of completeness. 
          Fix $\gamma\in(0,1)$ and recall that 
          \begin{align*}
          \mathcal{L}_1({\gamma})=\{t>\e: \mathcal{A}_1(t)< -\gamma(3\log t)^{1/3}\}.
          \end{align*} 
          Fix $\varepsilon\in (0, 1)$. According to Lemma \ref{lem:min}, 
          there exits a positive constant $C_\varepsilon$ such that for all sufficiently large $m$
          \begin{align*}
          \P\{\mathcal{L}_1(\gamma)\cap [m, m+1)\neq \emptyset\} &\leq \P\left\{\inf_{s\in [m,m+1)}\mathcal{A}_1(s) < -\gamma(3\log m)^{1/3} \right\}\\
          &= \P\left\{\min_{s\in [0,1]}\mathcal{A}_1(s) < -\gamma(3\log m)^{1/3} \right\}\\
          &\leq \frac{C_\varepsilon}{m^{(1-\varepsilon)\gamma^3}}. 
          \end{align*}  
          Therefore, we can cover $\mathcal{L}_1(\gamma)\cap [\e^n, \e^{n+1})$ by intervals of length $1$ to see that for all $\rho>0$ and sufficiently large $n$, 
          \begin{align*}
          \E\left[\nu_\rho^n(\mathcal{L}_1(\gamma))\right]&\leq \sum_{\substack{m\in\mathbb{Z}_+\\ [m,m+1)\subset [\e^n, \e^{n+1})}}\e^{-n\rho}
          \P\{\mathcal{L}_1(\gamma)\cap [m, m+1)\neq \emptyset\}\\
          &\leq C_\varepsilon \e^{-n\rho} \sum_{\substack{m\in\mathbb{Z}_+\\ [m,m+1)\subset [\e^n, \e^{n+1})}}m^{-(1-\varepsilon)\gamma^3}\\
          &\leq C_\varepsilon \e^{-n\rho} \e^{-n(1-\varepsilon)\gamma^3}\e^{n+1}= C_\varepsilon
          \e^{-n(\rho+(1-\varepsilon)\gamma^3-1)+1}.
          \end{align*} 
          Thus, 
          \begin{align*}
          \E\left[\sum_{n=0}^\infty\nu_\rho^n(\cl_1(\gamma)) \right] <\infty, \quad \text{if $\rho> 1-(1-\varepsilon)\gamma^3$}.
          \end{align*}  
          This proves that ${\rm Dim_{H}}(\cl_1(\gamma) )\leq \rho$ a.s. for all $\rho >1-(1-\varepsilon)\gamma^3$. Send $\rho\downarrow 1-(1-\varepsilon)\gamma^3$    and then $\varepsilon \downarrow 0$ to deduce that  ${\rm Dim_{H}}(\cl_1(\gamma)) \leq 1-\gamma^3$ a.s. This completes the proof of the upper bound in \eqref{eq:dimL1}.
\end{proof}

We next give a lower bound on the Hausdorff dimension of the lower level set of the Airy$_1$ process.

\begin{proof}[Proof of \eqref{eq:dimL1}: lower bound]
           Fix $\gamma\in (0, 1)$.
          We will prove that 
                              \begin{align*}
           {\rm Dim_{H}}\left\{t>\e: \mathcal{A}_1(t) <-\gamma\left(3\log t\right)^{1/3}\right\}\geq 1-\theta, \quad
           \text{a.s. for all $\theta\in(\gamma^3, 1)$}.
           \end{align*} 
           Choose and fix $\theta\in (\gamma^3, 1)$. We also choose and fixed sufficiently small positive constants $\delta$ and $\eta$ such that
           \begin{align}\label{exponent3}
           \theta-\delta-(1+3\eta)\gamma^3>0. 
           \end{align}

           We define the points $x_{i,n}$ and $z_{j, n}(i)$ in the same way as in \eqref{x_in} and \eqref{z_in} with $\theta$ and
           $\delta$
           satisfying the condition in \eqref{exponent3}. 
                                 For $i=1, \ldots, L_n$, 
           \begin{align}\label{eq:sum2}
           \P\left\{\inf_{s\in [x_{i,n}, x_{i,n}+\e^{n\theta})}\frac{\mathcal{A}_1(s)}{(\log s)^{1/3}} \geq 
           -\gamma3^{1/3}\right\} 
           &\leq            \P\left\{\inf_{s\in [x_{i,n}, x_{i,n}+\e^{n\theta})}\mathcal{A}_1(s)\geq 
           - \gamma\left(3n\right)^{1/3}\right\} \nonumber\\
           & \leq             \P\left\{\min_{1\leq j\leq \ell_n(i)}\mathcal{A}_1(z_{j,n}(i))\geq 
           - \left(3\gamma n\right)^{1/3}\right\}\nonumber\\
           &=   \tilde{\mathcal{P}}_{i,n,1} +    \tilde{\mathcal{P}}_{i,n,2},
           \end{align}
           where
           \begin{align*}
               \tilde{\mathcal{P}}_{i,n,1}&=      \P\left\{\min_{1\leq j\leq \ell_n(i)}\mathcal{A}_1(z_{j,n}(i))\geq 
           - \left(3\gamma n\right)^{1/3}\right\} - \prod_{j=1}^{\ell_n(i)}\P\left\{\mathcal{A}_1(z_{j,n}(i))\geq 
           - \left(3\gamma n\right)^{1/3}\right\},\nonumber \\
           \tilde{\mathcal{P}}_{i,n,2}&= \prod_{j=1}^{\ell_n(i)}\P\left\{\mathcal{A}_1(z_{j,n}(i))\geq 
           - \left(3\gamma n\right)^{1/3}\right\}.
           \end{align*}

           Because the Airy$_1$ process is associated and its the one-point distribution
            has a continuous probability density function,  we derive from \eqref{FKG2} that for $i=1, \ldots, L_n$
           \begin{align*}
            \tilde{\mathcal{P}}_{i,n,1} &\leq \sum_{1\leq j<k\leq \ell_n(i)}
                       \bigg( \P\left\{\mathcal{A}_1(z_{j,n}(i))\leq 
           -\gamma\left(3 n\right)^{1/3}, \mathcal{A}_1(z_{k,n}(i))\leq 
           -\gamma\left(3 n\right)^{1/3}\right\} \\
           &\qquad \qquad\qquad\qquad \qquad - \P\left\{\mathcal{A}_1(z_{j,n}(i))\leq 
           -\gamma\left(3 n\right)^{1/3}\right\}  \P\left\{\mathcal{A}_1(z_{k,n}(i))\leq 
           -\gamma\left(3 n\right)^{1/3}\right\} \bigg)\\
           &\qquad \leq K \sum_{1\leq j<k\leq  \ell_n(i)}\left[\Cov(\mathcal{A}_1(z_{j,n}(i))\,, \mathcal{A}_1(z_{k,n}(i)))\right]^{1/3},
           \end{align*}
           where the second inequality follows from \eqref{eq:prob}. 
           By \eqref{eq:cov} and the fact that $|z_{j,n}(i)-z_{k,n}(i)|\geq \e^{n\delta}$ for all $1\leq j<k\leq  \ell_n(i)$,
           there exists a constant $K_2>0$ such that for $i=1, \ldots, L_n$
           \begin{align}\label{eq:dif2}
           \tilde{\mathcal{P}}_{i,n,1} & \leq K_2\sum_{1\leq j<k\leq  \ell_n(i)} \e^{-\frac13|z_{j,n}(i)-z_{k,n}(i))|^3} \nonumber\\
           &\leq  K_2 \ell_n(i)^2 \e^{-\frac13\e^{3n\delta}}
           \leq 4K_2 \e^{2n(\theta-\delta)-\frac13\e^{3n\delta}},
           \end{align}
           where the last inequality holds by \eqref{eq:ln}.
           To estimate $\tilde{\mathcal{P}}_{i,n,2}$, we first see from that \eqref{eq:lowertail2}
           for all sufficiently large $n$, 
           \begin{align*}
           \P\left\{\mathcal{A}_1(0)\geq 
- \gamma\left(3n\right)^{1/3}\right\}  =         1-   \P\left\{\mathcal{A}_1(0)<
- \gamma\left(3n\right)^{1/3}\right\}\leq 1-\e^{-(1+3\eta)n\gamma^3},
           \end{align*}
           where $\eta$ is a fixed positive number satisfying \eqref{exponent3}.
           Then, by stationarity of the Airy$_1$ process,  for $i=1, \ldots, L_n$,
           \begin{align}\label{eq:tail4}
           \tilde{\mathcal{P}}_{i,n,2} &=\left(\P\left\{\mathcal{A}_1(0)\geq 
           - \gamma\left(3 n\right)^{1/3}\right\}\right)^{\ell_n(i)} \nonumber\\
           &\leq \left(1-\e^{-(1+3\eta)n\gamma^3}\right)^{\ell_n(i)}
           \leq \e^{-\ell_n(i)\e^{-(1+3\eta)\gamma^3 n}} \leq 
            \e^{-\e^{n(\theta-\delta)-(1+3\eta)\gamma^3 n}},
           \end{align}
           where the second inequality holds by the inequality $1-x\leq \e^{-x}$ for all $x\geq0$ 
          and the third inequality by \eqref{eq:ln}.
          
           Because  $L_n\leq \e^{1+n(1-\theta)}$, we deduce from 
            \eqref{eq:sum2}, \eqref{eq:dif2} and \eqref{eq:tail4} that
           \begin{align*}
           &\sum_{i=1}^{L_n}\P\left\{\inf_{s\in [x_{i,n}, x_{i,n}+\e^{n\theta})}\frac{\mathcal{A}_1(s)}{(\log x)^{1/3}} \geq 
           -\gamma3^{1/3}\right\} \\
           &\qquad \qquad \qquad \qquad\qquad\qquad
           \leq  4K_2 \e^{1+n(1-\theta)+2n(\theta-\delta)-\frac13\e^{3n\delta}} + \e^{1+n(1-\theta)-\e^{n(\theta-\delta)-(1+3\eta)\gamma^3 n}},
           \end{align*}
           which implies by \eqref{exponent3} that
                      \begin{align*}
           \sum_{n=1}^{\infty}\sum_{i=1}^{L_n}&\P\left\{\inf_{s\in [x_{i,n}, x_{i,n}+\e^{n\theta})}\frac{\mathcal{A}_1(s)}{(\log s)^{1/3}} \geq 
           -\gamma3^{1/3}\right\} <\infty.
\end{align*}
          Borel-Cantelli's lemma ensures that almost surely, for all but a finite number of integers $n\geq1$
          \begin{align*}
          \inf_{s\in [x_{i,n}, x_{i,n}+\e^{n\theta})}\frac{\mathcal{A}_1(s)}{(\log x)^{1/3}} <
           -\gamma3^{1/3}, \quad \text{for all $1\leq i\leq L_n$}.
          \end{align*}
          This means that the lower level set $\mathcal{L}(\gamma)$ is $\theta$-thick a.s. (see Definition \ref{def:thick}). Therefore, we conclude from Proposition \ref{prop:thick} that for $\theta\in (\gamma^3, 1)$
          \begin{align*}
          {\rm Dim_{H}}\left\{t>\e: \mathcal{A}_1(t) <-\gamma\left(3\log t\right)^{1/3}\right\}\geq1-\theta, \quad \text{a.s.}
          \end{align*}
          We let $\theta\downarrow\gamma^3$ to complete the proof. 
\end{proof}

\section{Macroscopic Hausdorff dimension of the Airy$_2$ process}
\label{s:HDU}

As we have seen in Section \ref{sec:th1}, the approach to Macroscopic Hausdorff dimension of the level sets of the Airy$_1$ process relies on the exponential decay rate of the covariance of the Airy$_1$ process. This method does not apply to the Airy$_2$ process since
           its covariance decays polynomially; see Widom \cite{Wid04}.  We will make use of the exponential LPP to  give some quantitative estimates on the tail probabilities of the 
           maximum and minimum of the Airy$_2$ process. First we state the estimates about the maximum and minimum of the Airy$_2$ process over an interval which can be obtained using the weak convergence of the point-to-point passage times in exponential LPP \eqref{e:airy2limit} and assuming the corresponding exponential LPP estimates. The exponential LPP estimates will be proved in Section \ref{s:ELPP}.

\begin{proposition}
    \label{p:Airy_2HDUL}
    For any $\vep, \delta> 0, x$ sufficiently large (depending on $\vep$) and $t$ sufficiently large (depending on $\vep,\delta$) there exists $c>0$ such that the following holds:
    \[
    \P \left(\sup_{s \in [0,t]} \ca_2(s) \leq x \right) \leq \e^{-\e^{-\frac 43 (1+\vep)x^{3/2}}t^{1-\delta}}+\e^{-c (\log t)^2}.
    \]
\end{proposition}
\begin{proof}
         This is an immediate consequence of Proposition \ref{t:UBELPP} below and Theorem \ref{t:lppairy2limit}.
\end{proof}

\begin{proposition}
    \label{p:HDLLB}
    For any $\vep, \delta> 0, x$ sufficiently large (depending on $\vep$) and $t$ sufficiently large (depending on $\vep,\delta$) there exist $c, c'>0$ such that the following holds:
    \[
    \P \left(\inf_{s \in [0,t]} \ca_2(s) \geq -x \right) \leq \left(\e^{-\left \{\e^{-\frac {1}{12} (1+\vep)x^{3}}-\e^{-c (\log t)^3}\right \}}\right)^{t^{1-\delta}}+\e^{-c' (\log t)^{3/2}}.
    \]
\end{proposition}
\begin{proof}
          This is a consequence of Proposition \ref{t:ELPPLB} below and Theorem \ref{t:lppairy2limit}.
\end{proof}

\subsection{Proof of (\ref{eq:dimU2}) and (\ref{eq:dimL2}) } 

\begin{proof}[Proof of \eqref{eq:dimU2}]
           First, according to \cite[Corollary 1.3]{CHH23}, there exists a positive constant $x_0$ such that
         \begin{align}\label{uniformA2tail}
          \mathrm{P}\left(\sup_{0\leq s\leq 2}|\mathcal{A}_2(s)-\mathcal{A}_2(0)| \geq x\right) \leq \e^{-\frac{x^2}{16}}, 
         \quad \text{for all $x>x_0$}.
         \end{align}
         Using triangle inequality and the fact that  (see \cite{BBD08, DuV13})
         \begin{align*}
         \P\{\mathcal{A}_2(0)>x\} = \e^{-(\frac43+o(1))x^{3/2}}, \quad \text{as $x\to+\infty$},
         \end{align*}
         we see that for any $\vep>0$,  there exists $C>0$ (depending on $\vep$) such that for sufficiently large $x$ (depending on $\vep$)
    \begin{align}\label{eq:A2max}
    \P \left(\sup_{s \in [0,1]} \ca_2(s) \geq x \right) \leq C\e^{-\frac{4}{3}(1-\vep)x^{3/2}}.
    \end{align}
     Then we can combine \eqref{eq:A2max} with \cite[Theorem 4.1]{KKX17} to obtain that for $\gamma\in (0,1)$
           \begin{align}\label{eq:dim-upper2}
           {\rm Dim_{H}}\left\{t>\e: \mathcal{A}_2(t) >\gamma\left((3\log t)/4\right)^{2/3}\right\}\leq 1-\gamma^{3/2}, \quad \text{a.s.}
           \end{align} 
           
           We next show that 
            for $\gamma\in (0,1)$
           \begin{align}\label{eq:dim-lower2}
           {\rm Dim_{H}}\left\{t>\e: \mathcal{A}_2(t) >\gamma\left((3\log t)/4\right)^{2/3}\right\}\geq 1-\gamma^{3/2}, \quad \text{a.s.}
           \end{align} 
        Fix $\theta \in (\gamma^{3/2},1)$.  Analogous to the proof of the lower bound for \eqref{eq:dim}, 
        we consider the points 
\[
x_{i,n}=\e^n+(i-1)\e^{n \theta},
\]
for $i=1,2, \dots, L_n:=\lfloor (\e-1)\e^{n(1-\theta)}+1\rfloor.$ Since for all $i=1,2,\dots, L_n$, $x_{i,n}+e^{n \theta} \leq e^{n+1}$, we have
\[\P \left \{\sup_{s \in [x_{i,n},x_{i,n}+\e^{n \theta})} \frac{\ca_2(s)}{(\log s)^{2/3}} \leq \gamma \left(\frac 34 \right)^{2/3}\right \} \leq \P \left \{\sup_{s \in [x_{i,n},x_{i,n}+\e^{n \theta})} \ca_2(s) \leq \gamma \left((n+1) \frac 34 \right)^{2/3} \right \}.\]
Using stationarity of the Airy$_2$ process and applying Proposition \ref{p:Airy_2HDUL} with $t=\e^{n\theta}$, we have that for any $\vep,\delta>0$ there exists $n_0\in \mathbb{N}$ (depending on $\vep,\delta, \theta$) such that for all $n\geq n_0$ and for $i=1,2, \dots, L_n$
\[
\P \left \{\sup_{s \in [x_{i,n},x_{i,n}+\e^{n \theta})} \frac{\ca_2(s)}{(\log s)^{2/3}} \leq \gamma \left(\frac 34 \right)^{2/3}\right \} \leq \e^{-\e^{-(1+\vep)(n+1) \gamma^{3/2}}\e^{n\theta(1-\delta)}}+\e^{-c' (n \theta)^2}.
\]
Therefore, 
\begin{align*}
&\sum_{i=1}^{L_n}\P \left \{\sup_{s \in [x_{i,n},x_{i,n}+\e^{n \theta})} \frac{\ca_2(s)}{(\log s)^{2/3}} \leq \gamma \left(\frac 34 \right)^{2/3}\right \} \leq L_n \left( \e^{-\e^{-(1+\vep)(n+1) \gamma^{3/2}}\e^{n\theta(1-\delta)}}+\e^{-c' (n \theta)^2}\right)\\& \leq \e^{1+n(1-\theta)}\left( \e^{-\e^{-(1+\vep)(n+1) \gamma^{3/2}}\e^{n\theta(1-\delta)}}+\e^{-c' (n \theta)^2}\right),
\end{align*}
where the last inequality follows from the fact that $L_n \leq \e^{1+n(1-\theta)}$.
Finally, as $\theta >\gamma^{3/2}$, one can choose $\vep$ and $\delta$ (depending on $\theta$) such that 
\[
\theta(1-\delta)-(1+\vep)\gamma^{3/2}>0.
\]
Then 
\[
\sum_{i=1}^{L_n}\P \left \{\sup_{s \in [x_{i,n},x_{i,n}+\e^{n \theta})} \frac{\ca_2(s)}{(\log s)^{2/3}} \leq \gamma \left(\frac 34 \right)^{2/3}\right \} \leq \e^{1+n(1-\theta)}\left( \e^{-\e^{c n}}+\e^{-c' (n \theta)^2}\right),
\]
for some $c>0.$ With this choice of $\vep$ and $\delta$ we see that 
\[
\sum_{n=1}^{\infty}\sum_{i=1}^{L_n}\P \left \{\sup_{s \in [x_{i,n},x_{i,n}+\e^{n \theta})} \frac{\ca_2(s)}{(\log s)^{2/3}} \leq \gamma \left(\frac 34 \right)^{2/3}\right \} < \infty.
\]
Therefore, by Borel-Cantelli lemma, almost surely for sufficiently large $n,$
\[
\left \{ \frac{\ca_2(s)}{(\log s)^{2/3}} \geq \gamma \left( \frac 34 \right)^{2/3} \right \} \cap [x_{i,n}, x_{i,n}+\e^{n \theta}) \neq \emptyset \text{ for all } 1 \leq i \leq L_n.
\]
Therefore, by Proposition \ref{prop:thick}, almost surely,
\[
{\rm Dim_{H}}\left\{t>\e: \mathcal{A}_2(t) >\gamma\left((3\log t)/4\right)^{2/3}\right\} \geq 1-\theta.
\]
As $\theta \in (\gamma^{3/2},1)$ is arbitrary, we obtain \eqref{eq:dim-lower2}. 
This completes the proof. 
\end{proof}

We move on to prove \eqref{eq:dimL2}.
\begin{proof}[Proof of \eqref{eq:dimL2}]
          First, similar to the proof of Lemma \ref{lem:min}, we can apply Lemma \ref{lemma:minlow} below and
           Theorem \ref{t:lppairy2limit} to deduce that 
           for any $\vep>0$,  there exists $C>0$ (depending on $\vep$) such that for sufficiently large $x$ (depending on $\vep$)
    \begin{align}\label{eq:minlow}
    \P \left(\min_{s \in [0,1]} \ca_2(s) \leq -x \right) \leq C\e^{-\frac{1}{12}(1-\vep)x^3}.
    \end{align}
    Then, we can combine \eqref{eq:minlow} with the same arguments as in the proof of the lower bound for \eqref{eq:dimL1} to derive
    that for $\gamma\in (0, 1)$, ${\rm Dim_{H}}(\mathcal{L}_2(\gamma))\leq 1-\gamma^{3}$ almost surely.
    
    It remains to prove that for $\gamma\in (0, 1)$
    \begin{align}\label{eq:lower2}
    {\rm Dim_{H}}(\mathcal{L}_2(\gamma))\geq 1-\gamma^{3}, \quad \text{a.s.}
    \end{align}
Fix $\theta \in (\gamma^3,1)$. Again, we consider the points 
\[
x_{i,n}=\e^n+(i-1)\e^{n \theta},
\]
for $i=1,2, \dots, L_n:=\lfloor (\e-1)\e^{n(1-\theta)}+1\rfloor$ and it is clear that 
\[\P \left \{\inf_{s \in [x_{i,n},x_{i,n}+\e^{n \theta})} \frac{\ca_2(s)}{(\log s)^{1/3}} \geq -\gamma \left(12\right)^{1/3}\right \} \leq \P \left \{\inf_{s \in [x_{i,n},x_{i,n}+\e^{n \theta})} \ca_2(s) \geq -\gamma \left((n+1) 12\right)^{1/3} \right \}.\]
Applying Proposition  \ref{p:HDLLB} with $t=\e^{n \theta}$, we see that for any $\vep,\delta>0$ there exists $n_0\in \mathbb{N}$ (depending on $\vep,\delta, \theta$) such that for all $n\geq n_0$ and for $i=1,2, \dots, L_n$
\[
\P \left \{\inf_{s \in [x_{i,n},x_{i,n}+\e^{n \theta})} \frac{\ca_2(s)}{(\log s)^{1/3}} \geq -\gamma \left(12\right)^{1/3}\right \} \leq \left(\e^{-\left \{\e^{-(1+\vep)(n+1) \gamma^{3}}-\e^{-c n^3 }\right \}}\right)^{\e^{n\theta(1-\delta)}}+\e^{-c' n^{3/2}}.
\]
Therefore, 
\begin{align*}
&\sum_{i=1}^{L_n}\P \left \{\inf_{s \in [x_{i,n},x_{i,n}+\e^{n \theta})} \frac{\ca_2(s)}{(\log s)^{1/3}} \geq -\gamma \left(12 \right)^{1/3}\right \} \leq L_n \left(\left(\e^{-\left \{\e^{-(1+\vep)(n+1) \gamma^{3}}-\e^{-c n^3 }\right\}}\right)^{\e^{n\theta(1-\delta)}}+\e^{-c' n^{3/2}}\right)\\& \qquad\qquad\leq \e^{1+n(1-\theta)}\left(\left(\e^{-\left \{\e^{-(1+\vep)(n+1) \gamma^{3}}-\e^{-c n^3 }\right\}}\right)^{\e^{n\theta(1-\delta)}}+\e^{-c' n^{3/2}}\right),
\end{align*}
where the last inequality follows from the fact that $L_n \leq \e^{1+n(1-\theta)}$.
Finally, as $\theta >\gamma^{3}$, one can choose $\vep$ and $\delta$ (depending on $\theta$) such that 
\[
\theta(1-\delta)-(1+\vep)\gamma^{3}>0.
\]
Then 
\[
\sum_{i=1}^{L_n}\P \left \{\inf_{s \in [x_{i,n},x_{i,n}+\e^{n \theta})} \frac{\ca_2(s)}{(\log s)^{1/3}} \geq -\gamma \left(12 \right)^{1/3}\right \} \leq \e^{1+n(1-\theta)}\left((\e^{-\e^{c n}}+\e^{-c' n^{3/2}}\right),
\]
for some $c>0.$
With this choice we see that 
\[
\sum_{n=1}^{\infty}\sum_{i=1}^{L_n}\P \left \{\inf_{s \in [x_{i,n},x_{i,n}+\e^{n \theta})} \frac{\ca_2(s)}{(\log s)^{1/3}} \geq -\gamma \left(12\right)^{1/3}\right \} < \infty.
\]
Therefore, by Borel-Cantelli lemma, almost surely for all sufficiently large $n,$
\[
\left \{ \frac{\ca_2(s)}{(\log s)^{1/3}} \leq -\gamma \left( 12\right)^{1/3} \right \} \cap [x_{i,n}, x_{i,n}+\e^{n \theta}) \neq \emptyset \text{ for all } 1 \leq i \leq L_n.
\]
Therefore, by Proposition \ref{prop:thick}, almost surely,
\[
{\rm Dim}_{\rm H}(\cl_2(\gamma)) \geq 1-\theta.
\]
As $\theta \in (\gamma^{3},1)$ is arbitrary, we obtain \eqref{eq:lower2}.
This completes the proof.
\end{proof}
\section{Estimates in Exponential LPP}
\label{s:ELPP}
In this section we state and prove all the exponential LPP estimates we have used in the previous sections. First we fix some notations. For $v\in \mathbb{Z}^2$, $T_{\mathcal{L}_0, v}$ denotes the last passage time between $v$ and $\mathcal{L}_0$. Let $I^{\delta}$ denote the interval of length $\floor{\delta(2N)^{2/3}}$
on $\mathcal{L}_{2N}$ with midpoint $(N,N)$. Let $I^{m, \delta}$ denote the interval of length  $\floor{\delta(2N)^{2/3}}$
on $\mathcal{L}_{2N}$ with midpoint $ (N-\floor{m(2N)^{2/3}}, N+\floor{m(2N)^{2/3}})$. The first lemma is a sharp upper bound for the event that the minimum line-to-point passage time over an interval is small.
\begin{lemma}[{{\cite[Lemma 2.7]{BaB24}}}] \label{lem:LPP}
          For any $\varepsilon >0$, there exists $\delta>0$  such that for any $x$ sufficiently large (depending on $\varepsilon$), for $N$ sufficiently large (depending on $x, \varepsilon$)
          \begin{align*}
          \P\left\{\min_{u\in I^\delta} T_{\mathcal{L}_0, u}-4N \leq -x 2^{4/3}N^{1/3}\right\} \leq \e^{-\frac16(1-\varepsilon)x^3}.
          \end{align*}
\end{lemma}
Next lemma is similar to the previous lemma for point-to-point passage time.
\begin{lemma}{\cite[Lemma 3.9]{BaB24}}
\label{lemma:minlow}For any $ \vep>0, m_0>0$ there exists $\delta>0$ (depending on $\vep$),such that for $m \in [0,m_0]$ and for all $x$ sufficiently large (depending on $\vep$), $N$ sufficiently large (depending on $m_0,\vep, x$) 
\[
\P \left( \min_{u_N(s) \in I^{m,\delta}  }\left(T_N(s)-4N+2^{4/3}N^{1/3}s^2\right) \leq -2^{4/3}N^{1/3}x\right) \leq \e^{-\frac{1}{12}(1-\vep) x^3}.
\]
\end{lemma}
  We proceed to give some estimates on  the tail probabilities of the maximum and minimum of the exponential LLP, which play crucial role in the study of Macroscopic Hausdorff dimension of the Airy$_2$ process. In particular, we prove Proposition \ref{t:UBELPP} and Proposition \ref{t:ELPPLB}. Their corresponding estimates for the Airy process were Proposition \ref{p:Airy_2HDUL} and Proposition \ref{p:HDLLB}. We start with an estimate on the expectation of last passage percolation. 
              \begin{lemma}
    \label{lem: expectation estimate}For any $\gamma>0$ with $\gamma < \frac mn < \gamma^{-1},$ there exists constant $C>0$ (depending only on $\gamma$) such that for all $m, n \geq 1$
\begin{equation}
\label{eq: expectation estimate}
|\E(T_{\bo{0},(m,n)})-( \sqrt{m}+\sqrt{n})^2| \leq Cn^{1/3}.
\end{equation}
\end{lemma}
\begin{proof}
    Note that by \cite[Theorem 2]{LR10} (see also \cite[Theorem 4.1]{BGZ21}) it follows that for all $(m,n)$ as in the statement there exist $C, c>0$ (depending on $\gamma$) such that for all $x>0$
    \begin{align}
    &\P \left( T_{\bo{0}, (m,n)}-( \sqrt{m}+\sqrt{n})^2 \geq x n^{1/3}\right) \leq C\e^{-c \min\{xn^{1/3},x^{3/2}\}},\label{eq:p1}\\
    &\P \left( T_{\bo{0}, (m,n)}-( \sqrt{m}+\sqrt{n})^2 \leq -x n^{1/3}\right) \leq C\e^{-c x^3}.\label{eq:p2}
    \end{align}
    Now
    \begin{align*}
        &\left \vert\frac{\E(T_{\bo{0},(m,n)})-( \sqrt{m}+\sqrt{n})^2}{n^{1/3}} \right \vert \leq \E \left \vert \frac{T_{\bo{0},(m,n)}-( \sqrt{m}+\sqrt{n})^2}{n^{1/3}}\right \vert \\
        &\quad = \int_{0}^{\infty} \P \left( \left \vert T_{\bo{0},(m,n)}-( \sqrt{m}+\sqrt{n})^2\right \vert \geq x n^{1/3}\right)\d x \\
        &\quad \leq \int_{0}^{\infty} \P \left( T_{\bo{0},(m,n)}-( \sqrt{m}+\sqrt{n})^2 \geq x n^{1/3}\right)\d x +\int_{0}^{\infty} \P \left( T_{\bo{0},(m,n)}-( \sqrt{m}+\sqrt{n})^2 \leq -x n^{1/3}\right)\d x
        \end{align*}
        Using the estimates in \eqref{eq:p1} and \eqref{eq:p2}, we see that for some $C,c>0$ depending on $\gamma$
        \begin{align*}
        \left \vert\frac{\E(T_{\bo{0},(m,n)})-( \sqrt{m}+\sqrt{n})^2}{n^{1/3}} \right \vert &\leq \int_{0}^\infty C\e^{-c\min \{x^{3/2} , xn^{1/3}\}}\d x
       +\int_{0}^{\infty} C\e^{-c x^3}\d x.
        \end{align*}
        Clearly, the two integrals on the right hand side are uniformly bounded in $n$.
\end{proof}
           

\begin{proposition}
    \label{t:UBELPP}
    For any $\vep, \delta>0, x$ sufficiently large (depending on $\vep$) and $t$ sufficiently large (depending on $\vep, \delta$), there exists $N(\vep,t, x) \in \N, c>0$ such that for all $N \geq N(\vep, t, x)$
    \[
    \P \left(\max_{s \in [0,t]} T_N(s)-4N+s^2 2^{4/3}N^{1/3} \leq x 2^{4/3}N^{1/3}  \right) \leq \e^{-\e^{-\frac 43(1+\vep)x^{3/2}}t^{1-\delta}}+\e^{-c(\log t)^2}.
    \]
\end{proposition}
\begin{proof}
          We adopt the strategy in  \cite[Theorem 3.2]{BaB24} to prove the above estimate. 
           For $j=1,2,\dots, \lfloor t^{1-\delta}\rfloor=\ell_t$, we consider the following points in $[0,t].$
    \[
    z_j:=(j-1)t^{\delta}.
    \]
    We have 
    \begin{align*}
        &\P \left(\sup_{s \in [0,t]} T_N(s)-4N+s^2 2^{4/3}N^{1/3} \leq x2^{4/3}N^{1/3}\right)\\
        &\qquad \leq \P \left(\bigcap^{\ell_t}_{j=1}\left \{ T_N(z_j)-4N+z_j^22^{4/3}N^{1/3} \leq x2^{4/3}N^{1/3} \right \} \right).
    \end{align*}
    Let 
    \[
    \ca:=\left \{\text{for all } j,  T_N(z_j)-4N+z_j^22^{4/3}N^{1/3} \leq x2^{4/3}N^{1/3} \right \}.
    \] 
    We chose $N$ sufficiently large depending on $t$. Without loss of generality assume that $\frac{N}{(\log t)^2}$ is an integer \footnote{Note that the correct definition should involve the floor or the ceiling function. But to avoid notational overhead we assume $\frac{N}{(\log t)^2}$ to be an integer. It can be checked easily that it does not affect the arguments in a non-trivial way. }. Let us consider the line $\cl_{\frac{2N}{(\log t)^2}}$. We define 
    $$v_j':= \left(\frac{N}{( \log t)^2}-\lfloor \frac{1}{(\log t)^2}z_j(2N)^{2/3}\rfloor, \frac{N}{( \log t)^2}+\lfloor \frac{1}{(\log t)^2}z_j(2N)^{2/3}\rfloor\right).$$
     For $j=1,2,\dots, \ell_t$, let $I_j$ and $J_j$ denote the intervals of length $(2N)^{2/3}$ on $\cl_{2N}$ and $\cl_{\frac{2N}{(\log t)^2}}$ respectively and $I_j$ has midpoint $u_N(z_j)$ and $J_j$ has midpoint $v_j'$. Let $P_j$ denote the parallelogram whose one pair of opposite sides lie on $I_j$ and $J_j$. We consider the following events. 
    \begin{align*}
    &\cb:= \left \{\text{for all } j, T_{\bo{0},v_j'}-\E \left( T_{\bo{0},v_j'}\right) \geq -\frac {\vep}{100} x 2^{4/3}N^{1/3}\right \},\\
    & \cc:=\left \{\text{for all } j,\widetilde{T}_{v_j', u_N(z_j)}-\E \left(T_{v_j',u_N(z_j)}\right) \leq x\left(1+\frac{\vep}{25} \right)2^{4/3}N^{1/3} \right \},
    \end{align*}
    where $\widetilde{T}_{v_j',u_N(z_j)}$ is the maximum passage time between $v_j'$ and $u_N(z_j),$ over all up/right paths restricted in the parallelogram $P_j.$ As for large enough $t$ (depending on $\delta$), $P_j$ are disjoint, the event $\cc$ is an intersection of independent events. We observe the following: by super-additivity and the fact that the restricted passage time is smaller than the actual passage time we have for all $j$
    \[
    T_N(z_j) \geq T_{\bo 0, v_j'}+\widetilde{T}_{v_j',u_N(z_j)}.
    \]
    Therefore, on the event $\mathcal{A},$
    \begin{align}
    & T_{\bo 0, v_j'}-\E \left( T_{\bo 0, v_j'}\right)+\widetilde{T}_{v_j',u_N(z_j)}-\E \left( T_{v_j',u_N(z_j)}\right)\nonumber \\
    & \qquad
    \label{eq:union bound}\leq 4N-z_j^2 2^{4/3}N^{1/3}+x2^{4/3}N^{1/3}- \left(\E \left( T_{\bo 0, v_j'}\right)+\E \left( T_{v_j',u_N(z_j)}\right) \right).
    \end{align}
 Using Lemma \ref{lem: expectation estimate} and the Taylor series expansion we get that 
\begin{align*}
\E \left( T_{\bo 0, v_j'}\right)&=\frac{2N}{(\log t)^2}+\frac{2N}{(\log t)^2} \sqrt{1-\frac{\left(\lfloor \frac{1}{(\log t)^2}z_j(2N)^{2/3}\rfloor\right)^2}{\frac{N^2}{(\log t)^4}}} + O(N^{1/3})\\
&=\frac{4N}{(\log t)^2}-\frac{N}{(\log t)^2}\frac{\left(\lfloor \frac{1}{(\log t)^2}z_j(2N)^{2/3}\rfloor\right)^2}{\frac{N^2}{(\log t)^4}}+O(N^{1/3})\\
&=\frac{4N}{(\log t)^2}-\frac{1}{(\log t)^2}z_j^22^{4/3}N^{1/3}+O(N^{1/3}).
\end{align*}
Now let us define $w_{j,n}:=\lfloor z_j(2N)^{2/3}\rfloor - \lfloor \frac{1}{(\log t)^2}z_j(2N)^{2/3}\rfloor$. Then, by \eqref{eq: expectation estimate}
\begin{align*}
    \E \left( T_{v_j',u_N(z_j)}\right)&=2N\left(1-\frac{1}{(\log t)^2}\right)+2N\left(1-\frac{1}{(\log t)^2}\right)\sqrt{1-\frac{w_{j,n}^2}{N^2 \left( 1-\frac{1}{(\log t)^2}\right)^2}}+ O(N^{1/3})\\
    &=4N \left(1-\frac{1}{(\log t)^2} \right)-z_j^2\left(1-\frac{1}{(\log t)^2} \right)2^{4/3}N^{1/3}+O(N^{1/3}).
\end{align*}
Thus we have that there exists a constant $C>0$ such that for sufficiently large $x$ (depending on $\varepsilon$) and $N$ sufficiently large (depending on $t$)
\[
\E \left( T_{\bo 0, v_j'}\right)+\E \left( T_{v_j',u_N(z_j)}\right) \geq 4N-z_j^22^{4/3}N^{1/3}-CN^{1/3} \geq 4N-z_j^22^{4/3}N^{1/3}-x\frac{3 \vep}{100}N^{1/3}.
\]
Hence, from \eqref{eq:union bound} we see that for sufficiently large $N$ (depending on $t$)
    \[
    \ca \subset \cb^c \cup \cc.
    \]
Now, we find upper bounds for $\P(\cb^c)$. In \cite[Proposition 1.4]{Jo00}, it was shown that the last passage time between $\bo{0}$ and $(m,n)$ is equal in distribution to the largest eigenvalue of Laguerre Unitary Ensemble. In \cite[Theorem 1.4, (ii)]{BBBK24}, sharp estimates for the lower tail of these largest eigenvalues were obtained. Therefore, using this correspondence between the last passage times and random matrices we also have the sharp estimates for the lower tail of passage times. Finally a union bound we have for sufficiently large $N$ (depending on $x,t,\vep$), sufficiently large $x, t$ (depending  and $\vep$)
\[
\P(\cb^c) \leq t^{1-\delta} \e^{-cx^3 (\log t)^2} \leq \e^{-c'(\log t)^2},
\]
where in the above $c$ can be chosen to be any fixed constant smaller than $\frac{1}{12}$ (this follows from \cite[Theorem 1.4, (ii)]{BBBK24}. Once the choice of $c$ is fixed, all the other parameters $x,t,N$ will depend on this choice. For the event $\mathcal{C}$ note that by \cite[Lemma 3.6]{BaB24} and independence, we get that for sufficiently large $N$ depending on $\vep, t$ and $x$ sufficiently large depending on $\vep$
\[
\P(\cc) \leq \left( 1-\e^{-\frac 43 (1+\vep)x^{3/2}}\right)^{t^{1-\delta}} \leq \e^{-\e^{-\frac 43 (1+\vep)x^{3/2}}t^{1-\delta}}.
\] 
This completes the proof.
\end{proof}
\begin{proposition}
      \label{t:ELPPLB}
       For any $\vep, \delta>0, x$ sufficiently large (depending on $\vep$) and $t$ sufficiently large (depending on $\vep, \delta$), 
       there exists $N(\vep, t, x) \in \N, c,c'>0$ such that for all $N \geq N(\vep,t, x)$
    \[
    \P \left(\min_{s \in [0,t]} T_N(s)-4N+s^2 2^{4/3}N^{1/3} \geq -x 2^{4/3}N^{1/3}  \right) \leq \left(\e^{-\left \{\e^{-\frac{1}{12}(1+\vep)x^{3}}-\e^{-c (\log t)^3}\right \}}\right)^{t^{1-\delta}}+\e^{-c'(\log t)^{3/2}}.
    \]
\end{proposition}
    \begin{proof} The proof is  similar to that of \cite[Theorem 3.8]{BaB24}.  We consider the points $z_j$ as defined in the proof of Proposition \ref{t:ELPPLB}. We have 
        \begin{align*}
        &\P \left(\min_{s \in [0,t]} T_N(s)-4N+s^2 2^{4/3}N^{1/3} \geq -x2^{4/3}N^{1/3} \right)\\
        &\qquad\leq \P \left(\bigcap^{\ell_t}_{j=1}\left \{ T_N(z_j)-4N+z_j^22^{4/3}N^{1/3} \geq -x2^{4/3}N^{1/3} \right \} \right).
    \end{align*}
     Let 
    \[
    \ca:=\left \{\text{for all } j,  T_N(z_j)-4N+z_j^22^{4/3}N^{1/3} \geq -x2^{4/3}N^{1/3} \right \}.
    \] 
    We chose $N$ sufficiently large depending on $t$. Let us consider the line $\cl_{\frac{2N}{(\log t)^3}}$ (as before for ease of notation we will assume that $\frac{N}{(\log t)^3}$ is an integer). We define
     $$v_j':= \left(\frac{N}{( \log t)^3}-\lfloor \frac{1}{(\log t)^3}z_j(2N)^{2/3}\rfloor, \frac{N}{( \log t)^3}+\lfloor \frac{1}{(\log t)^3}z_j(2N)^{2/3}\rfloor\right).$$
    For $j=1,2,\dots, \ell_t$ let $I_j$ and $J_j$ denote the intervals of length $\lfloor \log t(2N)^{2/3}\rfloor$ on $\cl_{2N}$ and $\cl_{\frac{2N}{(\log t)^3}}$ respectively and $I_j$ has midpoint $u_N(z_j)$ and $J_j$ has midpoint $v_j'$. Let $\widetilde{J}_j$ is the interval of length $\mu (2N)^{2/3}$ on $\cl_{\frac{2N}{(\log t)^3}}$ with midpoint $v_j'$ for some small enough $\mu$ which we will choose later depending on $\vep$. Let $P_j$ denote the parallelogram whose one pair of opposite sides lie on $I_j$ and $J_j$. We consider the following events. 
    \begin{align*}
    &\cb:=\left \{\text{for all } j,\max_{v \in \widetilde{J}_j} \{T_{\bo{0},v}-\E \left(T_{\bo{0},v} \right)\} \leq \frac{\vep}{100}x 2^{4/3}N^{1/3} \right \}.\\
    & \cc:= \{\text{for all } j, \Gamma_{u_N(z_j)} \cap \widetilde{J}_j^c =\emptyset\}.\\
    & \cd:=\{\text{for all }j,\Gamma_{u_N(z_j)} \cap P_j^c =\emptyset\}.\\
    & \ce:=\left \{\text{for all }j, \max_{v \in \widetilde{J}_j} \{\widetilde{T}_{v,u_N(z_j)}-\E \left(T_{v,z_j}\right)\} \geq -\left(1+\frac{\vep}{50}\right)x2^{4/3}N^{1/3} \right \},
    \end{align*}
    where $\widetilde{T}_{v,u_N(z_j)}$ is the maximum passage time over all paths restricted to $P_j$ between $v$ and $u_N(z_j)$. Note that as for sufficiently large $t$, the parallelograms $P_j$ are disjoint, $\ce$ is intersection of independent events. We observe that on the events $\cb \cap \cc \cap \cd \cap \ce^c,$ for all $j$ there exists some $\tilde{v}_j \in \tilde{J}_j$
    \[
    T_N(z_j)=T_{\bo{0},\tilde{v}_j}+\widetilde{T}_{\tilde{v}_j,u_N(z_j)}.
    \]
    From this, Lemma \ref{lem: expectation estimate} and similar calculations as we did to estimate the expectations in the proof of Proposition 4.2 we see that
    \[
    \cb \cap \cc \cap \cd \cap \ce^c \subset \ca^c.
    \]
    Hence,
   \[
    \ca \subset \cb^c \cup \cc^c \cup \cd^c \cup \ce.
    \]
    We first find an estimate for $\P \left(\cb^c \cup \cc^c \cup \cd^c \right).$ As consequence of \cite[Theorem 4.2, (ii)]{BGZ21} and a union bound we get there exists $c_1>0$ (depending on $\vep$) such that for sufficiently large $t$ depending on $\delta>0$, for $N$ sufficiently large (depending on $x$)
    \[
    \P \left(\cb^c \right) \leq t^{1-\frac \delta2}e^{-c_1 x^{3/2}(\log t)^{3/2}}
    \]
    Now if the event $\cc^c$ happens then for some $j,$ the transversal fluctuation of $\Gamma_{u_N(z_j)}$ on the line $\cl_{\frac{2N}{(\log t)^3}}$ is more than $\mu (\log t)^2 \left(\frac{2N}{(\log t)^3} \right)^{2/3}$. Due to \cite[Proposition 2.1, (i)]{BBB23} and a union bound we obtain that this event has small probability. Precisely, there exists $c_2>0$ (depending on $\mu$) such that for sufficiently large $t$ (depending on $\mu$) and sufficiently large $N$ (depending on $t$)
    \[
    \P\left( \cc^c \right) \leq t^{1-\delta}e^{-c_2 (\log t)^6}.
    \]
    Finally, if the event $\cd^c$ happens then for some $j$, the geodesic $\Gamma_{u_N(z_j)}$ goes out of the parallelogram $P_j$. Due to \cite[Proposition C.9]{BGZ21} and a union bound we have the following upper bound. There exists $c_3>0$ such that for all $t$ sufficiently large and $N$ sufficiently large depending on $t$
    \[
    \P \left( \cd^c\right) \leq t^{1-\delta}e^{-c_3(\log t)^3}.
    \]
    Combining all the above we get that for all $t$ sufficiently large (depending on $\vep, \mu$) and $N$ sufficiently large (depending on $t$) there exists $c>0$ (depending on $\mu, \vep$) such that 
    \[
    \P \left(\cb^c \cup \cc^c \cup \cd^c \right) \leq t^{1-\frac \delta2}\e^{-c(\log t)^{3/2}}.
    \]
    Now we consider the event $\ce.$ For a fixed $j$, we get that 
    \begin{align*}
    &\P \left(\max_{v \in \widetilde{J}_j} \{\widetilde{T}_{v,u_N(z_j)}-\E \left(T_{v,u_N(z_j)}\right)\} \geq -\left(1+\frac{\vep}{50}\right)x2^{4/3}N^{1/3}\right) \leq\\& \P \left( \max_{v \in \widetilde{J}_j} \{T_{v,u_N(z_j)}-\E \left(T_{v,u_N(z_j)}\right)\} \geq -\left(1+\frac{\vep}{50}\right)x2^{4/3}N^{1/3}\right)+\P \left( \mathsf{LTF}\right),
    \end{align*}
    where the event $\mathsf{LTF}$ is defined as follows:
    \[
    \mathsf{LTF}:=\{\text{there exists } v \in \widetilde{J}_j \text{ such that } \Gamma_{v,u_N(z_j)} \cap P_j^c \neq \emptyset\}.
    \]
    By ordering of geodesics (see \cite[Lemma 2.3]{BSS17B},\cite[Lemma 11.2]{BSS14}, \cite[Lemma 5.7]{H20}) we see that if the event $\mathsf{LTF}$ happens then certain geodesics will have large transversal fluctuation. Thus we apply \cite[Proposition C.9]{BGZ21} we see that for sufficiently large $N$ and $t$
    \[
    \P \left(\mathsf{LTF} \right) \leq \e^{-c (\log t)^3}.
    \]
    Finally, by \cite[Lemma 3.12]{BaB24}, for any $\vep>0$ there exists $\mu>0$ such that for sufficiently large $N$ (depending on $\vep,t,x$) and $x$ sufficiently large (depending on $\vep$), $t$ sufficiently large depending on $\vep$
\begin{align*}
&\P \left(\max_{v \in \widetilde{J}_j} \{\widetilde{T}_{v,u_N(z_j)}-\E \left(T_{v,u_N(z_j)}\right)\} \geq -\left(1+\frac{\vep}{50}\right)x2^{4/3}N^{1/3}\right)\\
&\qquad\qquad \qquad \leq 1-\e^{-\frac{1}{12}(1+\vep)x^3}+\e^{-c (\log t)^3} \leq \e^{-\left \{\e^{-\frac{1}{12}(1+\vep)x^3}-\e^{-c (\log t)^3}\right \}}.
\end{align*}
By independence we get 
\[
\P \left(\ce \right) \leq \left(\e^{-\left \{\e^{-\frac{1}{12}(1+\vep)x^3}-\e^{-c (\log t)^3}\right \}} \right)^{t^{1-\delta}}
\]
Combining all the above we get 
\[
\P \left(\ca \right) \leq \left(\e^{-\left \{\e^{-\frac{1}{12}(1+\vep)x^3}-e^{-c (\log t)^3}\right \}}\right)^{t^{1-\delta}}+\e^{-c'(\log t)^{3/2}}.
\]
This completes the proof.
\end{proof}

\noindent\textbf{Acknowledgment}.  SB thanks Riddhipratim Basu for useful discussion.  SB is supported by scholarship from National Board for Higher Mathematics (NBHM) (ref no: 0203/13(32)/2021-R\&D-II/13158).
FP is supported in part by National Key R\&D Program of China (No. 2022YFA 1006500) and National Natural Science Foundation of China (No. 12201047).

 \bigskip
 
 \begin{small}
\noindent\textbf{Sudeshna Bhattacharjee} Department of Mathematics, Indian Institute of Science, Bengaluru, India.\\
Email: \texttt{sudeshnab@iisc.ac.in}\\
\end{small}
 
 \begin{small}
\noindent\textbf{Fei Pu}
Laboratory of Mathematics and Complex Systems,
School of Mathematical Sciences, Beijing Normal University, 100875, Beijing, China.\\
Email: \texttt{fei.pu@bnu.edu.cn}\\
\end{small}


\begin{thebibliography}{999}

\bibitem{BBD08}
Baik, J., Buckingham, R. and DiFranco, J.:
Asymptotics of Tracy-Widom distributions and the total integral of a Painlev\'e II function. 
{\it Comm. Math. Phys.} {\bf 280} (2008), no. 2, 463--497.

\bibitem {BBB23} Balázs, M., Basu, R. and Bhattacharjee, S.: Geodesic trees in last passage percolation and some related problems. arXiv.2308.07312 (2024).

\bibitem{BaT89}
Barlow, M. T. and Taylor, S. J.:
Fractional dimension of sets in discrete spaces.
With a reply by J. Naudts
{\it J. Phys. A} {\bf 22} (1989), no.13, 2621--2628.


\bibitem{BaT92}
Barlow, M. T. and Taylor, S. J.:
Defining fractal subsets of $\mathbb{Z}^d$.
{\it Proc. London Math. Soc.  (3)} {\bf 64} (1992), no.1, 125--152.

\bibitem {BBBK24}Baslingker, J. and Basu, R. and Bhattacharjee, S. and Krishnapur, M.: Optimal tail estimates in $\beta$-ensembles and applications to last passage percolation. arXiv.2405.12215 (2024).

\bibitem{BBF23}
Basu, R., Busani, O. and Ferrari, P.L.:
On the exponent governing the correlation decay of the Airy$_1$ process. 
{\it Comm. Math. Phys.} {\bf 398} (2023), no. 3, 1171--1211.

\bibitem{BaB24}
Basu, R. and  Bhattacharjee, S.:
Limit theorems for extrema of Airy processes. arXiv:2406.11826 (2024)

\bibitem{BGZ21}Basu, R. and Ganguly, S. and Zhang, L.: Temporal correlation in last passage percolation with flat initial condition via {B}rownian comparison. {\it Comm. Math. Phys.} {\bf 383} (2021), no. 3, 1805--1888.
\bibitem{BSS17B}Basu, R., Sarkar, S. and Sly, A.: Coalescence of geodesics in exactly solvable models of last passage percolation. {\it J. Math. Phys.}{\bf 60} (2019), no. 9.
\bibitem{BSS14}Basu, R., Sidoravicius, V. and Sly, A.: Last passage percolation with a defect line and the solution of the slow bond problem. arxiv.1408.3464 (2016).

\bibitem{CHH23}Calvert, J., Hammond, A. and Hegde, M.:
Brownian structure in the KPZ fixed point.
{\it Ast\'erisque} No. 441 (2023)  v+119 pp.


\bibitem{DaP23}
Das, S. and Ghosal, P.:
Law of iterated logarithms and fractal properties of the KPZ equation. 
{\it Ann. Probab.}   {\bf 51} (2023), no. 3, 930--986.

\bibitem{DuV13}
Dumaz, L. and Vir\'ag, B.:
The right tail exponent of the Tracy-Widom $\beta$ distribution. 
{\it Ann. Inst. Henri Poincar\'e Probab\'e. Stat.} {\bf 49} (2013), no. 4, 915--933.


\bibitem{EPW67}
	Esary, J. D., Proschan, F. and Walkup, D. W.:
	Association of random variables with applications.
	{\it Ann.\ Math.\ Statist.}\ {\bf 38} (1967), no. 5, 1466--1474. 

\bibitem{GhY23}
Ghosal, P. and Yi, J.: 
Fractal geometry of the PAM in 2D and 3D with white noise potential. arXiv:2303.16063 (2023)
\bibitem{H20} Hammond, A.: Exponents governing the rarity of disjoint polymers in Brownian last passage percolation. {\it Proc. Lond. Math. Soc.} {\bf 120} (2020) no. 3, 370--433.

\bibitem{Jo00} Johansson, K. Shape fluctuations and random matrices. {\it Comm. Math. Phys.} {\bf 209}(2000) no. 2, 437--476.

\bibitem{KKX17}
Khoshnevisan, D., Kim, K. and Xiao, Y.:
Intermittency and multifractality: a case study via parabolic stochastic PDEs. {\it Ann. Probab.} {\bf 45} (2017), no.6A, 3697–3751.

\bibitem{KKX18}
Khoshnevisan, D., Kim, K. and  Xiao, Y.:
A macroscopic multifractal analysis of parabolic stochastic PDEs. 
{\it Comm. Math. Phys.}   {\bf 360} (2018), no. 1, 307--346.


\bibitem{Leb72}
Lebowitz, J. L.:
Bounds on the correlations and analyticity properties of ferromagnetic Ising spin systems.
{\it Comm. Math. Phys.} {\bf 28} (1972), 313--321.
 \bibitem{LR10} Ledoux, M., Rider, B.: Small deviation for {B}eta ensembles. {\it Electron. J. Probab.} {\bf 15} (2010), 1319--1343.

\bibitem{New80}
Newman, C. M.: Normal fluctuations and the FKG inequalities.
{\it Comm. Math. Phys.} {\bf 74} (1980), no. 2, 119--128.



\bibitem{PR12}
Prakasa Rao, B. L. S.:
Associated sequences, demimartingales and nonparametric inference.
{\it Probability and its Applications. Birkh\"auser/Springer}, Basel, (2012).

\bibitem{PSpng02}Pr\"ahofer, M. and Spohn, H.: Scale Invariance of the {PNG} Droplet and the {A}iry Process. {\it J. Stat. Phys. } {\bf 108} (2002) no. 5, 1071--1106.

\bibitem{Pu23} Pu, F.: Ergodicity, CLT and asymptotic maximum of the Airy$_1$ process. arXiv:2311.11217v3





\bibitem{Sas05}
Sasamoto, T.:
Spatial correlations of the 1D KPZ surface on a flat substrate.
{\it J. Phys. A} {\bf 38} (2005), no. 33, L549--L556.


\bibitem{WFS17}
Weiss, T., Ferrari, P. L. and Spohn, H.: Reflected Brownian motions in the KPZ universality class.
{ \it Springer Briefs in Mathematical Physics}, {\bf18}. Springer, Cham, 2017.



\bibitem{Wid04}
Widom, H.:
On asymptotics for the Airy process. 
{\it J. Statist. Phys.} {\bf 115} (2004), no. 3-4, 1129--1134.

\bibitem{Yi23}
Yi, J.:
Macroscopic multi-fractality of Gaussian random fields and linear stochastic partial differential equations with colored noise.
{\it J. Theoret. Probab.}   {\bf 36} (2023), no. 2, 926--947.

 \end{thebibliography}
\end{document}